\documentclass[12]{amsart}
\usepackage{amsmath,amssymb,amsbsy,amsfonts,amsthm,latexsym,amsopn,amstext,
                                               amsxtra,euscript,amscd}
\begin{document}

\newfont{\teneufm}{eufm10}
\newfont{\seveneufm}{eufm7}
\newfont{\fiveeufm}{eufm5}
%
%
\newfam\eufmfam
                \textfont\eufmfam=\teneufm \scriptfont\eufmfam=\seveneufm
                \scriptscriptfont\eufmfam=\fiveeufm
%
%
\def\frak#1{{\fam\eufmfam\relax#1}}
%


\def\bbbr{{\rm I\!R}} 
\def\bbbm{{\rm I\!M}}
\def\bbbn{{\rm I\!N}} 
\def\bbbf{{\rm I\!F}}
\def\bbbh{{\rm I\!H}}
\def\bbbk{{\rm I\!K}}
\def\bbbp{{\rm I\!P}}
\def\bbbone{{\mathchoice {\rm 1\mskip-4mu l} {\rm 1\mskip-4mu l}
{\rm 1\mskip-4.5mu l} {\rm 1\mskip-5mu l}}}
\def\bbbc{{\mathchoice {\setbox0=\hbox{$\displaystyle\rm C$}\hbox{\hbox
to0pt{\kern0.4\wd0\vrule height0.9\ht0\hss}\box0}}
{\setbox0=\hbox{$\textstyle\rm C$}\hbox{\hbox
to0pt{\kern0.4\wd0\vrule height0.9\ht0\hss}\box0}}
{\setbox0=\hbox{$\scriptstyle\rm C$}\hbox{\hbox
to0pt{\kern0.4\wd0\vrule height0.9\ht0\hss}\box0}}
{\setbox0=\hbox{$\scriptscriptstyle\rm C$}\hbox{\hbox
to0pt{\kern0.4\wd0\vrule height0.9\ht0\hss}\box0}}}}
\def\bbbq{{\mathchoice {\setbox0=\hbox{$\displaystyle\rm
Q$}\hbox{\raise 0.15\ht0\hbox to0pt{\kern0.4\wd0\vrule
height0.8\ht0\hss}\box0}} {\setbox0=\hbox{$\textstyle\rm
Q$}\hbox{\raise 0.15\ht0\hbox to0pt{\kern0.4\wd0\vrule
height0.8\ht0\hss}\box0}} {\setbox0=\hbox{$\scriptstyle\rm
Q$}\hbox{\raise 0.15\ht0\hbox to0pt{\kern0.4\wd0\vrule
height0.7\ht0\hss}\box0}} {\setbox0=\hbox{$\scriptscriptstyle\rm
Q$}\hbox{\raise 0.15\ht0\hbox to0pt{\kern0.4\wd0\vrule
height0.7\ht0\hss}\box0}}}}
\def\bbbt{{\mathchoice {\setbox0=\hbox{$\displaystyle\rm
T$}\hbox{\hbox to0pt{\kern0.3\wd0\vrule height0.9\ht0\hss}\box0}}
{\setbox0=\hbox{$\textstyle\rm T$}\hbox{\hbox
to0pt{\kern0.3\wd0\vrule height0.9\ht0\hss}\box0}}
{\setbox0=\hbox{$\scriptstyle\rm T$}\hbox{\hbox
to0pt{\kern0.3\wd0\vrule height0.9\ht0\hss}\box0}}
{\setbox0=\hbox{$\scriptscriptstyle\rm T$}\hbox{\hbox
to0pt{\kern0.3\wd0\vrule height0.9\ht0\hss}\box0}}}}
\def\bbbs{{\mathchoice
{\setbox0=\hbox{$\displaystyle     \rm S$}\hbox{\raise0.5\ht0\hbox
to0pt{\kern0.35\wd0\vrule height0.45\ht0\hss}\hbox
to0pt{\kern0.55\wd0\vrule height0.5\ht0\hss}\box0}}
{\setbox0=\hbox{$\textstyle        \rm S$}\hbox{\raise0.5\ht0\hbox
to0pt{\kern0.35\wd0\vrule height0.45\ht0\hss}\hbox
to0pt{\kern0.55\wd0\vrule height0.5\ht0\hss}\box0}}
{\setbox0=\hbox{$\scriptstyle      \rm S$}\hbox{\raise0.5\ht0\hbox
to0pt{\kern0.35\wd0\vrule height0.45\ht0\hss}\raise0.05\ht0\hbox
to0pt{\kern0.5\wd0\vrule height0.45\ht0\hss}\box0}}
{\setbox0=\hbox{$\scriptscriptstyle\rm S$}\hbox{\raise0.5\ht0\hbox
to0pt{\kern0.4\wd0\vrule height0.45\ht0\hss}\raise0.05\ht0\hbox
to0pt{\kern0.55\wd0\vrule height0.45\ht0\hss}\box0}}}}
\def\bbbz{{\mathchoice {\hbox{$\sf\textstyle Z\kern-0.4em Z$}}
{\hbox{$\sf\textstyle Z\kern-0.4em Z$}} {\hbox{$\sf\scriptstyle
Z\kern-0.3em Z$}} {\hbox{$\sf\scriptscriptstyle Z\kern-0.2em
Z$}}}}
\def\ts{\thinspace}

\newtheorem{theorem}{Theorem}
\newtheorem{lemma}[theorem]{Lemma}
\newtheorem{claim}[theorem]{Claim}
\newtheorem{cor}[theorem]{Corollary}
\newtheorem{prop}[theorem]{Proposition}
\newtheorem{definition}{Definition}
\newtheorem{question}[theorem]{Open Question}

\def\squareforqed{\hbox{\rlap{$\sqcap$}$\sqcup$}}
\def\qed{\ifmmode\squareforqed\else{\unskip\nobreak\hfil
\penalty50\hskip1em\null\nobreak\hfil\squareforqed
\parfillskip=0pt\finalhyphendemerits=0\endgraf}\fi}

\def\cA{{\mathcal A}}
\def\cB{{\mathcal B}}
\def\cC{{\mathcal C}}
\def\cD{{\mathcal D}}
\def\cE{{\mathcal E}}
\def\cF{{\mathcal F}}
\def\cG{{\mathcal G}}
\def\cH{{\mathcal H}}
\def\cI{{\mathcal I}}
\def\cJ{{\mathcal J}}
\def\cK{{\mathcal K}}
\def\cL{{\mathcal L}}
\def\cM{{\mathcal M}}
\def\cN{{\mathcal N}}
\def\cO{{\mathcal O}}
\def\cP{{\mathcal P}}
\def\cQ{{\mathcal Q}}
\def\cR{{\mathcal R}}
\def\cS{{\mathcal S}}
\def\cT{{\mathcal T}}
\def\cU{{\mathcal U}}
\def\cV{{\mathcal V}}
\def\cW{{\mathcal W}}
\def\cX{{\mathcal X}}
\def\cY{{\mathcal Y}}
\def\cZ{{\mathcal Z}}

\def\nrp#1{\left\|#1\right\|_p}
\def\nrq#1{\left\|#1\right\|_m}
\def\nrqk#1{\left\|#1\right\|_{m_k}}
\def\Ln#1{\mbox{\rm {Ln}}\,#1}
\def\nd{\hspace{-1.2mm}}
\def\ord{{\mathrm{ord}}}
\def\Cc{{\mathrm C}}
\def\Pb{\,{\mathbf P}}

\def\va{{\mathbf{a}}}

\newcommand{\comm}[1]{\marginpar{%
\vskip-\baselineskip 
\raggedright\footnotesize
\itshape\hrule\smallskip#1\par\smallskip\hrule}}




\newcommand{\ignore}[1]{}

\def\vec#1{\mathbf{#1}}

\def\e{\mathbf{e}}



\def\GL{\mathrm{GL}}

\hyphenation{re-pub-lished}

\def\rank{{\mathrm{rk}\,}}
\def\dd{{\mathrm{dyndeg}\,}}

\def\A{\mathbb{A}}
\def\B{\mathbf{B}}
\def \C{\mathbb{C}}
\def \F{\mathbb{F}}
\def \K{\mathbb{K}}
\def \Z{\mathbb{Z}}
\def \P{\mathbb{P}}
\def \R{\mathbb{R}}
\def \Q{\mathbb{Q}}
\def \N{\mathbb{N}}
\def \Z{\mathbb{Z}}

\def \nd{{\, | \hspace{-1.5 mm}/\,}}

\def\Zn{\Z_n}

\def\Fp{\F_p}
\def\Fq{\F_q}
\def \fp{\Fp^*}
\def\\{\cr}
\def\({\left(}
\def\){\right)}
\def\fl#1{\left\lfloor#1\right\rfloor}
\def\rf#1{\left\lceil#1\right\rceil}
\def\e{\mbox{\bf{e}}}
\def\ed{\mbox{\bf{e}}_{d}}
\def\ek{\mbox{\bf{e}}_{k}}
\def\eM{\mbox{\bf{e}}_M}
\def\emd{\mbox{\bf{e}}_{m/\delta}}
\def\eqk{\mbox{\bf{e}}_{m_k}}
\def\ep{\mbox{\bf{e}}_p}
\def\eps{\varepsilon}
\def\er{\mbox{\bf{e}}_{r}}
\def\et{\mbox{\bf{e}}_{t}}
\def\Kc{\,{\mathcal K}}
\def\Ic{\,{\mathcal I}}
\def\Bc{\,{\mathcal B}}
\def\Rc{\,{\mathcal R}}

\setlength{\textheight}{43pc}
\setlength{\textwidth}{28pc}

\title[Polynomial Dynamics and Nonlinear Pseudorandom Numbers]
{Multivariate Permutation Polynomial Systems and Nonlinear Pseudorandom Number Generators}

\author{Alina~Ostafe}
\address{Institut f\"ur Mathematik, Universit\"at Z\"urich,
Winterthurerstrasse 190 CH-8057, Z\"urich, Switzerland}
\email{alina.ostafe@math.uzh.ch}


\begin{abstract}In this paper we study a class 
of dynamical systems generated by iterations of multivariate permutation polynomial systems  
which lead to polynomial 
growth of the degrees of these iterations. Using these estimates and the same techniques studied previously for inversive generators, we 
bound exponential sums along the 
orbits of these dynamical systems and show that they admit 
much stronger estimates ``on average''
over all initial values 
$\vec{v} \in \F_p^{m+1}$
than in the general case and thus can be of use for pseudorandom number 
generation. 
  \end{abstract}

\maketitle


\section{Introduction}

Let $\cF = \{f_0, \ldots ,f_{m}\}$ be a system of $m+1$ polynomials in $m+1$ variables over an arbitrary field.  One can naturally define a  dynamical system
generated by its iterations, 
see~\cite{EvWa,Silv1}
and references therein 
for various aspects of such dynamical systems, and consider the orbits obtained by such 
iterations evaluated at a certain initial value 
$\(v_{0}, \ldots, v_{m}\)$. 
The statistical uniformity of the distribution 
(measured by the discrepancy) of one and  multidimensional 
nonlinear polynomial generators over a finite field 
have been studied in~\cite{GNS,GG,NiSh4,NiWi,TopWin}. However, almost all
previously known results are nontrivial only for those
polynomial generators
that produce sequences of extremely large period, which could
be hard to achieve in practice (the only known exceptions are generators
from  inversions~\cite {NiSh3}, power functions~\cite{FrSh},
Dickson polynomials~\cite{GoGuSh} and Redei functions~\cite{GutWin}). The reason
behind this is that typically the degree of iterated polynomial
systems grows exponentially, and that in all previous results
the saving over the trivial bound has been logarithmic.
Furthermore, it is easy to see that in the one-dimensional
case (that is, for $m=0$) the exponential growth of 
the degree of iterations of a nonlinear polynomial is
unavoidable. One also expects the same behaviour 
in the multidimensional case for ``random'' polynomials
$f_0, \ldots ,f_{m}$. However, as we saw in~\cite{OS}, for some specially selected
polynomials $f_0, \ldots ,f_{m}$ the degree may grow 
significantly slower.

In~\cite{OS} we describe a rather wide class of 
polynomial systems with polynomial growth of the degree
of their iterations. As a result we obtain much better estimates
of exponential sums, and thus of the discrepancy, for vectors
generated by these iterations
(after scaling them to the unit cube),
with a saving over
the trivial bound  being a power of $p$.

Obtaining stronger results ``on average''
over all initial values 
$\vec{v} \in \F_p^{m+1}$ 
is an interesting and challenging 
question. We remark that in the case of the so-called inversive generator rather stronger estimates ``on average" are available (see~\cite{NiSh3}) and also estimates for the average distribution of powers and primitive elements of the inversive generators are considered in~\cite{CW}. In this paper we study this problem by following the same arguments introduced for the inversive generator in~\cite{NiSh3}. For this we define a special family of multivariate polynomial
systems of~\cite{OS}, which beside the polynomial degree growth also leads to {\it permutation polynomial systems\/}. In turn this allows us to use the approach of~\cite{NiSh3} to obtain a
stronger bound on the discrepancy ``on average" over initial values.

Furthermore, here we exploit  the special structure of iterations of the polynomial
systems of~\cite{OS} that allows us to replace the use of the Weil
bound (see~\cite[Chapter~5]{LN}) 
by a more elementary and stronger estimate on the corresponding 
exponential sums which in turn 
leads to a better final result and for more general
systems of congruences.   
In fact, since our construction 
can easily be extended to polynomials over commutative rings,
the new estimate can also be used to study polynomials
maps over residue rings (while the Weil bound does not apply
there). 
This estimate 
can also be used to improve and generalise the main result of~\cite{OS}.

Finally, we note that we also hope that our results may be of use 
for some applications in polynomial dynamical systems. 

Throughout the paper, the implied constants in the symbols `$O$' and `$\ll$'
may occasionally, where obvious, depend on
some integer parameter $s\ge 1$ and are absolute otherwise. We recall that the notations $A = O(B )$ and $A\ll B$ are all equivalent to the assertion that the inequality $|A|\leq c|B|$ holds for some
constant $c > 0$.

\section{Permutation Polynomial Dynamical System with Slow Degree Growth}
\label{sec:sec2}
\subsection{General construction}

We recall and modify 
the construction of~\cite{OS} of multivariate 
polynomial systems with slow degree growth. Let  $\F$ be an arbitrary field  
and let the polynomials $g_i,h_i\in\F[X_{i+1}, \ldots ,X_m]$, $i=0,\cdots,m-1$, satisfying the following conditions: each polynomial $g_i$ 
has a {\it unique leading monomial\/} $X_{i+1}^{s_{i,i+1}}\ldots X_m^{s_{i,m}}$, that is,
\begin{equation}
\label{eq:Cond1}
g_i(X_{i+1},\ldots,X_m) = X_{i+1}^{s_{i,i+1}}\ldots X_m^{s_{i,m}} + 
\widetilde{g_i}(X_{i+1},\ldots,X_m), 
\end{equation}
where 
\begin{equation}
\label{eq:Cond2}
\deg_{X_j} \widetilde{g_i} <  s_{i,j}, 
\qquad \deg_{X_j} h_i\le  s_{i,j},
\end{equation}
for $i=0,\ldots,m-1$, $j = i+1, \ldots, m$.

Throughout, we use $\deg$ to denote the total degree of a multivariate polynomial.

We construct now a system $\cF = \{f_0, \ldots ,f_m\}$
of $m+1$ polynomials in 
the ring 
$\F[X_0, \ldots ,X_m]$ defined in the following way:
\begin{equation}
\label{eq:Polys}
\begin{split}
f_0(X_0, \ldots ,X_m)& = X_0g_0(X_1,\ldots,X_m)+h_0(X_1,\ldots,X_m),\\
f_1(X_0, \ldots ,X_m)&=X_1g_1(X_2,\ldots,X_m)+h_1(X_2,\ldots,X_m),\\
  &\ldots  \\
f_{m-1}(X_0, \ldots ,X_m)&=X_{m-1}g_{m-1}(X_m)+h_{m-1}(X_m),\\
f_m(X_0, \ldots ,X_m) &=aX_m+b,
\end{split}
\end{equation}
where $$
a,b\in\F,\quad a \ne 0,\quad \textrm{and} \quad g_i,h_i\in\F[X_{i+1},\ldots,X_m], \quad i=0,\ldots,m-1,
$$ 
are defined as above.

For each $i=0, \ldots ,m$ we define the $k$-th iteration of the polynomials $f_i$ by the recurrence relation
\begin{equation}
\label{eq:PolyIter}
f_i^{(0)}=X_i, \qquad f_i^{(k)}= f_i(f_0^{(k-1)}, \ldots ,f_m^{(k-1)}), \qquad k=0,1, \ldots\,  .
\end{equation}
The following result shows the exact form of the polynomials $f_i^{(k)}$ and also the polynomial growth of the degrees of the polynomials $X_ig_i$, $i=0,\ldots,m$, under iterations.

\begin{lemma}
\label{lem:LinTerm+Deg} Let $f_0, \ldots, f_m\in \F[X_0,\ldots,X_m]$ be  as in~\eqref{eq:Polys},
satisfying the conditions~\eqref{eq:Cond1} and \eqref{eq:Cond2}. 
Then for the polynomials $f_i^{(k)}$, $k=1,2,\ldots$, given by~\eqref{eq:PolyIter} we have
$$
f_i^{(k)} = X_i g_{i,k} (X_{i+1},\ldots,X_m) + h_{i,k}(X_{i+1},\ldots,X_m)
$$
where $ g_{i,k}, h_{i,k} \in \F[X_{i+1},\ldots,X_m]$ and
\begin{eqnarray*}
\deg g_{i,k}&=&\frac{1}{(m-i)!}k^{m-i}s_{i,i+1}\ldots s_{m-1,m}+\psi_i(k),\qquad  i=0,\ldots, m-1,\\
\deg g_{m,k}&=&0,
\end{eqnarray*}
where $\psi_i(T) \in \Q[T]$ is a polynomial of degree  $\deg \psi_i <m-i$.
\end{lemma}
\begin{proof}
We have
$$
f_i^{(k)}  = f_i^{(k-1)} g_i\(f_{i+1}^{(k-1)}, \ldots, f_{m}^{(k-1)}\)
+h_i\(f_{i+1}^{(k-1)}, \ldots, f_{m}^{(k-1)}\).
$$
Thus an easy inductive argument implies that 
$$
f_i^{(k)} = X_i g_{i,k}  (X_{i+1},\ldots,X_m) + 
h_{i,k} (X_{i+1},\ldots,X_m)
$$
for some polynomials $g_{i,k}, h_{i,k}  \in \F[X_{i+1},\ldots,X_m]$, 
with $\deg g_{i,k} \ge \deg h_{i,k}$, 
where  $i =0, \ldots, m$,  $k=1, 2,\ldots$. 

For the asymptotic formulas for the degrees of the polynomials $g_{i,k}$ 
see~\cite[Lemma~1]{OS} where it is given in the equivalent form 
for  $\deg f_i^{(k)}= \deg g_{i,k} + 1$.
\end{proof}

   \subsection{Permutation polynomial systems}
   
In order to be able to apply the technique introduced in~\cite{NiSh3} for inversive pseudorandom number generators, we need to work with systems of multivariate polynomials in $\Fp[X_0,\ldots,X_m]$ which induce maps that permute the elements of $\Fp^{m+1}$. Lidl and Niederreiter~\cite{LN, LN1} call such  systems 
{\it orthogonal polynomial systems\/}, but we here refer to
them as {\it permutation polynomial systems\/}.
      
Let the polynomial system $\cF = \{f_0, \ldots ,f_m\}$, $m\geq 1$, be defined by~\eqref{eq:Polys}
 and satisfy the conditions~\eqref{eq:Cond1} and \eqref{eq:Cond2}. It is obvious that this system is a permutation system if and only if the polynomials $g_i$, $i=0,\ldots,m$, 
 do not have zeros over $\Fp$. 

We note that a ``typical'' absolute irreducible polynomial in $m\ge 2$ variables 
over $\Fp$ always has lots of zeros. By a special case of the Lang-Weil theorem~\cite{LaWe} a polynomial $F$ in $m\ge 2$ variables over $\Fp$ always has $rp^{m-1}+O(p^{m-3/2})$ zeros where $r$ is the number of absolutely irreducible factors of $F$ (with the implied constant depending only on $\deg F$), see also~\cite{Sch}. That is why we seek ``atypical'' polynomials, as the example below shows.

One of the attractive choices of polynomials 
which would lead to a fast PRNG is
$$g_i(X_{i+1},\ldots,X_m) = \prod_{j=1}^{m-i} (X_{i+j}^2-a_{i,j})$$
and
$$  h_i(X_{i+1},\ldots,X_m) = b_{i}
$$
where $a_{i,j}$ are quadratic nonresidues and $b_i$ are any constants in $\Fp$.

Even simpler, one can take 
$$g_i(X_{i+1},\ldots,X_m) =   (X_{i+1}^2-a_i) $$
where $a_{i}$ are quadratic nonresidues.

\section{Polynomial Pseudorandom Number Generators}

\subsection{Construction}

Let $\cF=\{f_0,\ldots,f_m\}$ be a permutation polynomial system in $\F_p[X_0,\ldots,X_m]$ defined as in Section~\ref{sec:sec2}. 
We fix a vector $\vec{v} \in \F_p^{m+1}$ and
  consider the sequence defined by a recurrence congruence
modulo a prime $p$ of the form
\begin{equation}
\label{eq:Gen}
u_{n+1,i}\equiv f_i(u_{n,0},\ldots,u_{n,m})\!\!\! \pmod p, \qquad n =
0,1,\ldots,
\end{equation}
with the {\it initial values\/}
$(u_{0,0},\ldots,u_{0,m}) = \vec{v}$. 
We also assume that $0 \le u_{n,i} < p$, $i=0,\ldots,m$, $n=0, 1, \ldots$.

In particular, for any $n,k\ge 0$ and $i=0,\ldots,m$ we have
\begin{equation}
\label{eq:vecf}
u_{n+k,i}(\vec{v})=f_i^{(k)}(u_{n,0}(\vec{v}),\ldots,u_{n,m}(\vec{v})).
\end{equation}
Using the following vector notation
$$
\vec{u}_n(\vec{v})=(u_{n,0}(\vec{v}),\ldots,u_{n,m-1}(\vec{v}))
$$
we have the recurrence relation
$$
\vec{u}_{n+k}(\vec{v})=(f_0^{(k)}(u_{n,0}(\vec{v}),\ldots,u_{n,m}(\vec{v})),\ldots,f_{m-1}^{(k)}(u_{n,0}(\vec{v}),\ldots,u_{n,m}(\vec{v}))).
$$

We show that for almost all initial values 
$\vec{v} \in \F_p^{m+1}$, 
the sequence 
\begin{equation}
\label{eq:seq}
\(\frac{u_{n,0}(\vec{v})}{p}, \ldots, 
\frac{u_{n,m-1}(\vec{v})}{p}\),
\qquad n = 0,\ldots, N-1,
\end{equation}
is uniformly distributed for all $N\geq (\log p)^{2 + \varepsilon}$,
 any fixed $\varepsilon>0$ and sufficiently large~$p$.

\subsection{Exponential Sums}
\label{sec:ExpSum}

We put
$$\e_m(z) = \exp(2 \pi i z/m).$$

Our second main tool is the following bound on exponential sums which is stronger than the one immediately implied by the Weil bound
(see~\cite[Chapter~5]{LN}).


\begin{lemma}
\label{lem:Elem}
Let $f_0, \ldots, f_m\in \Fp[X_0,\ldots,X_m]$ be  as in~\eqref{eq:Polys},
 satisfying the conditions~\eqref{eq:Cond1} and \eqref{eq:Cond2}.
If $s_{0,1} \ldots s_{m-1,m} \ne 0$,  then there is a positive integer $k_0$ depending only on the degrees of the polynomials in $\cF$ such that for any integers  $k>l\geq k_0$ 
 and 
any nonzero $\vec{a} = (a_0, \ldots, a_{m-1}) \in \Fp^m$, 
for the polynomial
$$
F_{\vec{a},k,l} =  \sum_{i=0}^{m-1} \!a_i
\(f_i^{(k)}-f_i^{(l)}\), 
$$
we have
$$
\left|\sum_{x_0, \ldots, x_m =1 }^p
\ep\( F_{\vec{a}, k,l}(x_{0},\ldots,x_{m})\) \right| \ll k^m p^{m}.
$$
\end{lemma}

\begin{proof}
Let $s \le m-1$ be the smallest integer such that $a_s\ne 0$.
By Lemma~\ref{lem:LinTerm+Deg} we have
\begin{equation*}
\begin{split}
&\sum_{x_0, \ldots, x_m =1 }^p  
\ep\( F_{\vec{a}, k,l}(x_{0},\ldots,x_{m})\) \\
 &\qquad =\sum_{x_0, \ldots, x_m =1 }^p
\ep\( \sum_{i=0}^{m-1} \!a_i\(x_i(g_{i,k}-g_{i,l})+(h_{i,k}-h_{i,l})\)\)\\
 &\qquad =p^{s}\sum_{x_s, \ldots, x_{m} =1 }^p
\ep\( \sum_{i=s}^{m-1} \!a_i\(x_i(g_{i,k}-g_{i,l})+(h_{i,k}-h_{i,l})\)\)\\
&\qquad =p^{s}\sum_{x_{s+1}, \ldots, x_{m} =1 }^p
\ep\(h_{s,k}-h_{s,l}+ \sum_{i=s+1}^{m-1} \!a_i\(x_i(g_{i,k}-g_{i,l})+(h_{i,k}-h_{i,l})\)\)\\
&\qquad\qquad\qquad \qquad \qquad\qquad\qquad
\sum_{x_s=1}^p\ep\(  a_s x_s(g_{s,k}-g_{s,l})\).
\end{split}
\end{equation*}
 Then the sum over the variable $x_s$ is nonzero only 
if its coefficient 
$$
g_{s,k}(x_{s+1},\ldots,x_m)-g_{s,l}(x_{s+1},\ldots,x_m) \equiv 0 \pmod p,
$$ 
see~\cite[Equation~(5.9)]{LN1}.

We see from   Lemma~\ref{lem:LinTerm+Deg} that if $k> l\ge k_0$
for a sufficiently large $k_0$ then $g_{s,k}-g_{s,l}$
is a nontrivial polynomial modulo $p$ of degree $O(k^{m-s})=O(k^m)$. 
A simple inductive argument shows that a nontrivial
modulo $p$ polynomial 
in $r$ variables of degree $D$ may have only $O(Dp^{r-1})$
zeros modulo $p$, which concludes the proof. 
%
%
\end{proof}

We note that we do not include the linear polynomials $f_m^{(k)}$ and $f_m^{(l)}$ 
in $F_{\vec{a},k,l}$  as generally 
speaking in this case such  a linear combination may vanish 
even for nontrivial coefficients (note that it is possible that 
$f_m^{(k)}=f_m^{(l)}$ for $k \ne l$).

We follow the  scheme previously introduced in~\cite{NiSh3} for estimating the exponential sum introduced below, and thus the discrepancy of a sequence of points.

For a vector $\vec{a} = (a_0, \ldots,
a_{m-1}) \in \Fp^m$  and
integers $c,M,N$ with $M \ge 1$ and $N \ge 1$, we introduce
$$
V_{\vec{a},c}(M,N) =  \sum_{v_{0},\ldots, v_{m}\in \F_{p}}
\left|\sum_{n=0}^{N-1}
\ep \(\sum_{j=0}^{m-1} a_j f_j^{(n)}(v_0,\ldots,v_m)\) \eM(c n)\right|^2.
$$
Note that as in Lemma~\ref{lem:Elem} we do not include polynomials
$f_m^{(n)}$ in the above exponential sum.

\begin{lemma}
\label{lem:expsum} Let  the permutation polynomial system of $m+1$ polynomials 
$\cF=\{f_0,\ldots,f_m\} \in \F_p[X_{0},\ldots,X_{m}]$ 
of total degree $d \ge 2$ of the form~\eqref{eq:Polys}, 
satisfying the conditions~\eqref{eq:Cond1} and \eqref{eq:Cond2}.
Then for any  positive integers $c,M,N$ and any nonzero vector $\vec{a} = (a_0, \ldots,
a_{m-1}) \in \Fp^m$
we have
$$
V_{\vec{a},c}(M,N)  \ll A(N,p),
$$
where
$$
A(N,p) =
\left\{ \begin{array}{ll}
N p^{m+1}  & \mbox{if}\ N \le p^{1/{(m+1)}}, \\
N^2 p^{m(m+2)/(m+1)} & \mbox{if}\ N > p^{1/{(m+1)}}.
\end{array} \right.
$$
\end{lemma}


\begin{proof} 
We have
\begin{eqnarray*}
\lefteqn{V_{\vec{a},c}(M,N)  =   \sum_{k,l=0}^{N-1} \eM(c (k-l))}\\
& &\qquad \qquad \sum_{{v_{0},\ldots, v_{m}\in \F_{p}}}
\ep \(\sum_{j=0}^{m-1} a_j\(f_j^{{(k)}}(v_0,\ldots,v_m) -
f_j^{{(l)}}(v_0,\ldots,v_m)\)\) \\
 && \quad \le \sum_{k,l=0}^{N-1} \left| \sum_{{v_{0},\ldots, v_{m}\in \F_{p}}} \ep \( \sum_{j=0}^
{m-1} a_j \(f_j^{{(k)}}(v_0,\ldots,v_m) -f_j^{{(l)}}(v_0,\ldots,v_m )\)\) \right|.
\end{eqnarray*}
For  $O(N)$ values of $k$ and $l$ which are equal,  we estimate  the inner sum 
trivially by $p^{m+1}$.


For the other values, 
by Lemma~\ref{lem:Elem}  getting the
upper bound $O(N^mp^m)$ for the inner sum
for at most $N^2$ sums.
Hence,
\begin{equation}
\label{eq:expsum}
V_{\vec{a},c}(M,N) \ll N p^{m+1} + N^{m+2} p^m.
\end{equation}
Because $\cF$ is a permutation  polynomial system and using~\eqref{eq:vecf}, for any integer $L$ we obtain 
\begin{eqnarray*}
\lefteqn{\sum_{{v_{0},\ldots, v_{m}\in \F_{p}}}
\left|\sum_{n=L}^{L+N-1}
\ep \(\sum_{j=0}^{m-1} a_jf_j^{{(n)}}(v_0,\ldots,v_m)\) \eM(c n)\right|^2}\\& & =
   \sum_{{v_{0},\ldots, v_{m}\in \F_{p}}} \\
   & &\qquad \ \left|\sum_{n=0}^{N-1} \ep \(\sum_{j=0}^{m-1}
a_j f_j^{{(n)}}\(f_0^{(L)}(v_0,\ldots,v_m),\ldots,f_m^{(L)}(v_0,\ldots,v_m)\)\) \eM(cn) \right|^2\\ & &= \sum_{v_{0},\ldots, v_{m}\in \F_{p}}
\left|\sum_{n=0}^{N-1}
\ep \(\sum_{j=0}^{m-1} a_j f_j^{(n)}(v_0,\ldots,v_m)\) \eM(c n)\right|^2= V_{\vec{a},c}(M,N).
\end{eqnarray*}
Therefore, for any positive integer $K\le N$, separating the inner sum into
at most $N/K + 1$ subsums of length at most $K$, and using~\eqref{eq:expsum}, we derive
$$
V_{\vec{a},c}(M,N) \ll (K p^{m+1} + K^{m+2} p^m)N^2K^{-2} = N^2(K^{-1} p^{m+1} + K^{m} p^m).
$$
Thus, selecting $K = \min\{N, \fl{p^{1/{(m+1)}}}\}$ and taking into account
that $ N^{-1} p^{m+1} \geq N^{m} p^m$ for $N \le p^{1/{(m+1)}}$, we obtain the desired result. 
\end{proof}
Note that the estimates for $V_{\vec{a},c}(M,N)$ work not only over prime fields, but also over any finite field.

We also need
the identity (see~\cite{IwKow})
\begin{equation}
\label{Ident}
\sum_{-(m-1)/2 \le a  \le m/2} \e_m(ab ) =
\left\{ \begin{array}{ll}
0& \mbox{if}\ b\not \equiv 0 \pmod m, \\
m& \mbox{if}\ b \equiv 0 \pmod m.
\end{array} \right.
\end{equation}
Then we have the following inequality
\begin{equation}
\label{Inequal}
\sum_{r=L+1}^{L+Q} \e_m(cr)
\ll \min\left\{Q,\frac{m}{|c|}\right\}
\ll  \min\left\{m,\frac{m}{|c|}\right\}
\ll \frac{m}{|c|+1}
\end{equation}
which holds for any integers $c$, $Q$ and $L$ with $|c|\leq m/2$, and $m \ge Q \ge 1$,
see~\cite[Bound~(8.6)]{IwKow}.
 
\subsection{Discrepancy}
\label{sec:discr}

Given a sequence $\Gamma$ of $N$ points 
\begin{equation}
\label{eq:GenSequence}
\Gamma = \left\{(\gamma_{n,0}, \ldots, \gamma_{n,s-1})_{n=0}^{N-1}\right\}
\end{equation}
in the $s$-dimensional unit cube $[0,1)^s$
it is natural to measure the level of its statistical uniformity 
in terms of the {\it discrepancy\/} $\Delta(\Gamma)$. 
More precisely, 
$$
\Delta(\Gamma) = \sup_{B \subseteq [0,1)^s}
\left|\frac{T_\Gamma(B)} {N} - |B|\right|,
$$
where $T_\Gamma(B)$ is the number of points of  $\Gamma$
inside the box
$$
B = [\alpha_1, \beta_1) \times \ldots \times [\alpha_{s}, \beta_{s})
\subseteq [0,1)^s
$$
and the supremum is taken over all such boxes, see~\cite{DrTi,KuNi}.

We recall that the discrepancy is a widely accepted 
quantitative measure 
of uniformity of distribution of sequences, and thus good pseudorandom
sequences should (after an appropriate scaling) have a small discrepancy,
see~\cite{Nied1,Nied2}.

For an integer vector $\va = (a_0, \ldots, a_{s-1}) \in \Z^s$ we put
\begin{equation*}
|\va| = \max_{j = 0, \ldots, s-1} |a_j|, \qquad
r(\va) = \prod_{j=0}^{s-1} \max\{|a_j|, 1\}.
\end{equation*}

Typically the bounds on the discrepancy of a 
sequence  are derived from bounds of exponential sums
with elements of this sequence. 
The relation is made explicit in 
 the celebrated {\it Erd\H os-Turan-Koksma
inequality\/}, see~\cite[Theorem~1.21]{DrTi},
which we  present in the following form.

\begin{lemma}
\label{lem:Kok-Szu} For any
integer $L > 1$ and any  sequence $\Gamma$ of $N$ points~\eqref{eq:GenSequence}
the discrepancy $\Delta(\Gamma)$
satisfies the following bound:
$$
\Delta(\Gamma)< O \( \frac{1}{L}
+ \frac{1}{N}\sum_{ 0 < |\va| \le L} {1 \over r(\va)}
\left| \sum_{n=0}^{N-1} \exp \( 2 \pi i\sum_{j=0}^{s-1}a_j\gamma_{n,j} \)
\right| \).
$$
\end{lemma} 

Now, as in \cite{NiSh3}, combining Lemma~\ref{lem:Kok-Szu} with the bound 
obtained in Lemma~\ref{lem:expsum} we obtain stronger estimates for the discrepancy ``on average'' over all initial values.

\begin{theorem}
\label{thm:Discr} 
Let $0<\varepsilon<1$ and let the sequence $\{\vec{u}_n\}$ be given by~\eqref{eq:Gen}, where
the permutation system of $m+1$ polynomials 
$\cF=\{f_0,\ldots,f_m\} \in \F_p[X_{0},\ldots,X_{m}]$ 
of total degree $d \ge 2$ is  of the form~\eqref{eq:Polys}, 
satisfying the conditions~\eqref{eq:Cond1} and \eqref{eq:Cond2}, 
and such that $s_{0,1} \ldots s_{m-1,m} \ne 0$. Then for all initial values $\vec{v}\in \Fp^{m+1}$  except at most
$O(\varepsilon p^{m+1})$ of them, 
and any positive integer $N \le p^{m+1}$, 
the discrepancy $D_N(\vec{v})$ of the sequence~\eqref{eq:seq}
satisfies the bound 
$$
D_N(\vec{v})  \ll \varepsilon^{-1} B(N,p),
$$
where
$$
B(N,p) =
\left\{ \begin{array}{ll}
N^{-1/2}(\log N)^{m+1} \log p & \mbox{if}\ N \le p^{1/{(m+1)}}, \\
p^{-1/2(m+1)} (\log N)^{m+1} \log p & \mbox{if}\ N > p^{1/{(m+1)}}.
\end{array} \right.
$$
\end{theorem}

\begin{proof} Without loss of generality we can assume that $N \ge 2$.
 From Lemma~\ref{lem:Kok-Szu} with $G = \lfloor N/2 \rfloor$ we derive
$$
D_N(\vec{v}) \ll  {1 \over N}
+ \frac{1}{N}\sum_{ 0 < |\va| \le N/2} {1 \over r(\va)}
\left| \sum_{n=0}^{N-1} \ep \(\sum_{j=0}^{m-1}a_j u_{n,j}(\vec{v})\)
\right|.
$$
Let $m_\nu = 2^\nu$, $\nu =0, 1, \ldots$, and
define $k \ge 1$ by the condition $m_{k-1} <  N \le m_k$.
From~\eqref{Ident} we derive
\begin{eqnarray*}
\lefteqn{
\sum_{n=0}^{N-1} \ep\(\sum_{j=0}^{m-1}a_j u_{n,j}(\vec{v})\)}\\
& &=
\frac{1}{m_k} \sum_{n=0}^{m_k-1} \ep\(\sum_{j=0}^{m-1}a_j u_{n,j}(\vec{v})\)
\sum_{-(m_k-1)/2 \le c \le m_k/2} \sum_{r =0}^{N-1} \eqk(c(n-r)).
\end{eqnarray*}
Since $m_k/2 = m_{k-1}$, from~\eqref{Inequal} we obtain
\begin{eqnarray*}
\lefteqn{
\left| \sum_{n=0}^{N-1} \ep\(\sum_{j=0}^{m-1}a_j u_{n,j}(\vec{v})\)
\right|}\\
& &\ll   \sum_{|c|\le m_{k-1}} \frac{1}{|c|+1}
\left| \sum_{n=0}^{m_k-1} \ep\(\sum_{j=0}^{m-1}a_j u_{n,j}(\vec{v})\)
\eqk(cn)
\right|.
\end{eqnarray*}
It follows that
\begin{equation}
\label{Bound}
D_N(\vec{v}) \ll \Delta_k(\vec{v}),
\end{equation}
where
\begin{eqnarray*}
\Delta_k(\vec{\vec{v}}) &  = &
{1 \over N}
+ {1 \over m_k }\sum_{ 0 < |\va| \le m_{k-1}} {1 \over r(\va)}
 \sum_{|c|\le m_{k-1}}\frac{1}{|c|+1}\\
& & \qquad \qquad \qquad \qquad \cdot
\left| \sum_{n=0}^{m_k-1} \ep\(\sum_{j=0}^{m-1}a_j u_{n,j}(\vec{v})\)
\eqk(cn)\right|.
\end{eqnarray*}
Now
\begin{eqnarray*}
\lefteqn{\sum_{\vec{v} =({v_{0},\ldots, v_{m}) \in \F_{p}^{m+1}}} \Delta_k(\vec{v})  = \frac{p^{m+1}}{N}
+ {1 \over m_k }\sum_{ 0 < |\va| \le m_{k-1}} {1 \over r(\va)}
 \sum_{|c|\le m_{k-1}}  \frac{1}{|c|+1}}\\
& & \qquad \qquad \qquad \qquad \cdot
\sum_{{v_{0},\ldots, v_{m}\in \F_{p}}}
\left| \sum_{n=0}^{m_k-1} \ep\(\sum_{j=0}^{m-1} a_{j}f_j^{(n)}(v_0,\ldots,v_m)\)
\eqk(cn)\right|.
\end{eqnarray*}
Applying the Cauchy inequality, from Lemma~\ref{lem:expsum} we derive
$$
\sum_{{v_{0},\ldots, v_{m}\in \F_{p}}}
\left| \sum_{n=0}^{m_k-1} \ep\(\sum_{j=0}^{m-1} a_{j}f_j^{(n)}(v_0,\ldots,v_m)\)
\eqk(cn)\right| \ll p^{(m+1)/2} A(m_k,p)^{1/2}.
$$
Therefore
\begin{equation*}
\begin{split}
\sum_{{v_{0},\ldots, v_{m}\in \F_{p}}} 
\Delta_k(\vec{v})  
& \ll  \frac{p^{m+1}}{N} + {p^{(m+1)/2} A(m_k,p)^{1/2} \over m_k }\
\sum_{ 0 < |\va| \le m_{k-1}} {1\over r(\va)}\sum_{|c|<m_{k-1}}
\frac{1}{|c|+1}\\
& \ll   \frac{p^{(m+1)/2} A(m_k,p)^{1/2} (\log m_k)^{m+1}}{m_k },
\end{split}
\end{equation*}
where we used the standard bound for partial sums of the harmonic series
in the last step. Thus, for each $k = 1, \ldots, \rf{\log (p^{m+1})}$, the inequality
\begin{equation}
\label{e10}
\Delta_k(\vec{v})  
  \ge    { A(m_k,p)^{1/2} (\log m_k)^{m+1}
\log p \over \varepsilon m_k p^{(m+1)/2} }
  =  \varepsilon^{-1} B(m_k,p)
\end{equation}
can hold for at most $O(\varepsilon p^{m+1} / \log p)$
values of $v_{0},\ldots, v_{m}\in \F_{p}$. Therefore
the number of $v_{0},\ldots, v_{m}\in \F_{p}$ for which
\eqref{e10} holds for
at least one  $k = 1, \ldots, \rf{\log (p^{m+1})}$
is $O(\varepsilon p^{m+1})$. For all other $v_{0},\ldots, v_{m}$, 
we get from~\eqref{Bound},
$$
D_N(\vec{v}) \ll \Delta_k(\vec{v})
 < \varepsilon^{-1} B(m_k,p)  \ll
\varepsilon^{-1} B(N,p)
$$
for $1 \le N \le p^{m+1}$, where we used $m_k = 2m_{k-1} < 2N$
in the last step.
\end{proof}


\section{Remarks and Open Questions}

As we have mentioned, one of the attractive choices of polynomials~\eqref{eq:Polys}, 
which leads to a very fast 
pseudorandom number generator is
$$g_i(X_{i+1},\ldots,X_m) = X_{i+1}^2-a_i\qquad \text{and}
\qquad h_i(X_{i+1},\ldots,X_m) = b_{i}
$$
for some quadratic nonresidues $a_i$ and any constants $b_i$, $i = 0, \ldots, m-1$. 
The corresponding sequence of vectors is generated
at the cost of two multiplications per component. 
This naturally leads to a question of studying in what cases the 
periods of such sequences generated by such polynomial
dynamical systems are maximal.

We also note that it is natural to consider the 
joint distribution of several consecutive vectors
$$
\(\vec{u}_n(\vec{v}), \ldots, \vec{u}_{n+s-1}(\vec{v})\), 
\qquad n =0,1,\ldots\,,
$$
in the $sm$-dimensional space. It seems that the scheme used in~\cite{OS} 
can be also applied  to derive  such a result.

\section*{Acknowledgement}

The author would like to thank Igor Shparlinski for introducing the idea of using permutation polynomial systems in order to obtain better results and also
 Markus Brodmann, Joachim Rosenthal, Arne Winterhof and the anonymous referee for valuable comments and 
providing additional references.

The idea  of this work appeared during the
``Cryptography Retrospective Meeting'' at the  Fields Institute,
May, 2009,
which hospitality and financial support is gratefully  appreciated.

During the preparation of this paper,  
the author was also supported in part by 
the Swiss National Science Foundation   Grant 121874.

\end{document}